\documentclass{amsart}
              [1997/02/02 v1.2e PROC Author Class]

\copyrightinfo{2006}
  {American Mathematical Society}

\newtheorem{theorem}{Theorem}[section]
\newtheorem{corollary}[theorem]{Corollary}

\newtheorem{proposition}[theorem]{Proposition}

\newtheorem{question}[theorem]{Question}
\newtheorem{remark}[theorem]{Remark}

\begin{document}

\title[Baire classification and Namioka property]{Baire classification of separately continuous functions and Namioka property}

\author{V.V.Mykhaylyuk}
\address{Department of Mathematics\\
Chernivtsi National University\\ str. Kotsjubyn'skogo 2,
Chernivtsi, 58012 Ukraine}
\email{vmykhaylyuk@ukr.net}



\commby{Ronald A. Fintushel}



\begin{abstract}
We prove the following two results.

1. If $X$ is a completely regular space such that for every
topological space $Y$ each separately continuous function
$f:X\times Y\to\mathbb R$ is of the first Baire class, then every
Lindel\"{o}f subspace of $X$ bijectively continuously maps onto a
separable metrizable space.

2. If $X$ is a Baire space, $Y$ is a compact space and $f:X\times
Y\to\mathbb R$ is a separately continuous function which is a
Baire measurable function, then there exists a dense in $X$
$G_{\delta}$-set $A$ such that $f$ is jointly continuous at every
point of $A\times Y$ (this gives a positive answer to a question
of G.~Vera).
\end{abstract}

\maketitle
\section{Introduction}

A function $f:X\to\mathbb R$, which defined on a topological space $X$, is called {\it a function of the first Baire class}, if there exists a sequence $(f_n)^{\infty}_{n=1}$ of continuous functions $f_n:X\to\mathbb R$ such that $f(x)=\lim\limits_{n\to\infty}f_n(x)$ for every $x\in X$. For every at most countable ordinal $\alpha$ a function $f:X\to\mathbb R$ is called {\it a function of the $\alpha$ Baire class}, if there exists a sequence $(f_n)^{\infty}_{n=1}$ of functions $f_n:X\to\mathbb R$ of Baire class $<\alpha$ such that $f(x)=\lim\limits_{n\to\infty}f_n(x)$ for every $x\in X$. A function $f:X\to\mathbb R$ is called {\it Baire measurability} if $f$ is a function of $\alpha$ Baire class for some at most countable ordinal $\alpha$.

A function $f:X\to\mathbb R$ is called {\it a function of the first Lebesgue class}, if $f^{-1}(G)$ is a $F_{\sigma}$-set in $X$ for every open set $G\subseteq \mathbb R$.

Investigations of Baire and Lebesgue classification of separately continuous functions (that is functions of many variables which are continuous with respect to every variable) were started by A.Lebesgue in the classical paper [1]. These investigations were continued in papers of many mathematicians (see, for example,  [2] the literature given there).

In particular, W.~Moran and H.Rosental [3,4] proved that for a compact $X$ the following conditions are equivalent:

(i)\,\,\,for every compact space $Y$ every separately continuous function $f:X\times Y\to\mathbb R$ is a first Baire class function;

(ii)\,\,\,$X$ has the countable chain condition that is every system of pairwise disjoint open in $X$ nonempty set is at most countable.

A topological space $X$ is called {\it a Moran (weakly Moran) space}, if for every compact space $Y$ each separately continuous function $f:X\times Y\to\mathbb R$is a first Baire class function (a Baire measurable function). These notions were introduced in [5], where the following result was obtained.

\begin{theorem}\label{th:1.1} Let $X$ be a completely regular space with the countable chain condition, $Y$ be a compact, $f:X\times Y\to \mathbb R$ be a separately continuous function and $\varphi:Y\to C_p(X)$, $\varphi(y)(x)=f(x,y)$. Then the following conditions are equivalent:

$(i)$\,\,\,\,$f$ is a first Baire class function;

$(ii)$\,\,\,the space $\varphi(Y)$ is metrizable.
\end{theorem}

For a topological space $X$ by $C_p(X)$ we denote the space of all continuous functions $z:X\to\mathbb R$ with the topology of pointwise convergence.

Moreover, relations between Moran spaces and Namioka property  was investigated in [5]. A mapping $f:X\times Y\to\mathbb R$ which is defined on the product $X\times Y$ of topological spaces $X$ and $Y$ {\it has Namioka property}, if there exists an everywhere dense in $X$ $G_{\delta}$-set $A\subseteq X$ such that $f$ is jointly continuous at every point of set $A\times Y$. A topological space $X$ is called {\it a Namioka space}, if for every compact space $Y$ each separately continuous functions $f:X\times Y\to\mathbb R$ has the Namioka property. The following question was formulated in [5].

\begin{question}\label{q:1.2} Is every Baire Moran space a Namioka space?
\end{question}

A topological space $X$ is called {\it a space with $B$-property ($L$-property)}, if for every topological space $Y$ each separately continuous function $f:X\times Y\to\mathbb R$ is a first Baire (Lebesgue) function.

A standard reasoning (see [6, p.394]) shows that every first Baire class function is a first Lebesgue class function. Therefore every space with $B$-property has $L$-property. Note that spaces with $B$-property were called by Rudin spaces in [7]. But this term is not very successful because W.~Rudin in [8] used partitions of the unit for proof that every separately continuous mappings which is defined on the product of a topological space and a metrizable space and valued in a locally convex space is a Baire first mapping. W.~Rudin is not considered a mapping with valued in $\mathbb R$ separately.

A development of Rudin's result leads to the following notions.

A topological space $X$ is called {\it a $PP$-space}, if there exist a sequence $$\left((h_{n,i}:i\in I_n)\right)_{n=1}^{\infty}$$ of locally finite partitions of the unit $(h_{n,i}:i\in I_n)$ on $X$ and sequence $(\alpha_n)_{n=1}^{\infty}$ of families $\alpha_n=(x_{n,i}:i\in I_n)$ of points $x_{n,i}\in X$ such that for every $x\in X$ and neighborhood $U$ of $x$ in $X$ there exists an integer $n_0\in\mathbb N$ such that $x_{n,i}\in U$, if $n\geq n_0$ and $x\in {\rm supp}h_{n,i}$, where by ${\rm supp}h$ we denote {\it the support} $\{x\in X: h(x)\ne 0\}$ of $h$.

These notion was introdused in [9]. It was obtained in [10, Theorem 1] that every $PP$-space has the $B$-property.

A topological space $X$ with the topology ${\mathcal T}$ is called {\it quarter-stratifiable}, if there exists a function $g:\mathbb N\times X\to
{\mathcal T}$ such that

$(i)\,\,\,\,X=\bigcup\limits_{x\in X} g(n,x)$ for every $n\in \mathbb N$;

$(ii)\,\,\,$ if $x\in g(n,x_n)$ for every $n\in \mathbb N$, then $x_n\to x$.

If there exists a weaker metrizable topology ${\mathcal T}$ on a qarter-stratifiable space $X$ such that all covers ${\mathcal U}_n=\{g(n,x):x\in X\}$ can be chosen ${\mathcal T}$-open, then the space $X$ is called {\it metrically quarter-stratifiable}.

These notions were introduced in [7]. Moreover, it was shown in [7] (Theorem 6.2(3)) that every metrically quarter-stratifiable space has the $B$-property.

In the connections with an investigation of properties of spaces with the $B$-property the following questions were formulated in [7, question 6.3].

\begin{question}\label{q:1.3} Is there a space with the $B$-property which is not quarter-stratifiable? Has a (compact) space $X$ the $B$-property if the set $\{(x,y): y(x)=0\}$ is a $G_{\delta}$-set in $X\times C_p(X)$? Is a compact space with $B$-property metrizable?
\end{question}

Clearly that analogical questions can be formulated for spaces with $L$-property. In particular, since it follows from ç [11, Proposition 2.1] that every quarter-stratifiable space has $L$-property, the following question arises naturally.

\begin{question}\label{q:1.4}
Is there a space with the $L$-property which is not quarter-stratifiable?
\end{question}

Note that it was shown in [7]([12]) that every topological space $X$ has the $B$-property ($L$-property) if and only if the calculation function $c_X:X\times C_p(X)\to\mathbb R$, $c_X(x,y)=y(x)$, is a first Baire (Lebesgue) function. Since the space $C_p(X)$ the countable chain condition, it follows from Theorem \ref{th:1.1} that every compact subspace of a completely regular space with the $B$-property is metrizable (this gives a positive answer to the third part of Question \ref{q:1.3}. In other words, every compact subspace of completely regular space with the $B$-property is submetrizable, that is it can be continuously bijectively mapped on a metrizable space. Moreover, every Lindel$\rm\ddot{o}$f space with the $B$-property is separable [7, Theorem 6.2.(6)]. Generally, every Lindel$\rm\ddot{o}$f subspace $M$ of completely regular space $X$ with $L$-property contained in the closure of some countable set $A\subseteq X$ [13, Proposition
4.7]. Therefore, the following question arises naturally.

\begin{question} \label{q:1.5} Let $Z$ be completely regular space with the $B$-property and $X\subseteq Z$ be a Lindel$\rm\ddot{o}$f subspace of $Z$. Does there exist a continuous bijection which defined on $X$ and valued in a separable metrizable space?
\end{question}

In this paper using dependence on some coordinates of function we give positive answers to Questions 1.2 and 1.5.

\section{Some properties of $B$-spaces and $L$-spaces}

A set $G\subseteq X$ in a topological space $X$ is called {\it functionally open}, if there exists a continuous function $\varphi:X\to [0,1]$ such that $G=\varphi^{-1}((0,1])$, and a set $E\subseteq X$ is a {\it functionally $G_{\delta}$-set}, if $E=\bigcap\limits_{n\in\mathbb N} G_n$, where $(G_n)^{\infty}_{n=1}$ is a sequence of fuctionally open in $X$ sets. The sets $X\setminus G$ and $X\setminus E$ are called {\it functionally closed} and a {\it functionally $F_{\sigma}$set} respectively.

The following propositions refine characteristic properties of spaces with $B$-property and $L$-property [7,
Theorem 6.2.(1)] and [12, Proposition 2.3], which give the possibility to consider the calculation function only.

\begin{proposition}\label{p:2.1} Let $X$ be a topological space. Then the following conditions are equivalent:

$(i)$\,\,\, $X$ has the $B$-property;

$(ii)$\,\,\,the set $E=\{(x,y): y(x)=0\}$ is a functionally $G_{\delta}$-set in $X\times C_p(X)$.
\end{proposition}

\begin{proof} $(i) \Rightarrow (ii)$. Since $X$ has the $B$-property, there exists a sequence $(f_n)^{\infty}_{n=1}$ of continuous functions $f_n:X\times Y\to \mathbb R$, where $Y=C_p(X)$, which pointwise on $X$ converges to the separately continuous calculation function $f=c_X$. For every $m,n\in \mathbb N$ we put $G_{mn}=f_n^{-1}((-\frac{1}{m},\frac{1}{m}))$ and $E_{mn}=\bigcup\limits_{k\geq n} G_{mk}$. Since the sets $E_{mn}$ are functionally open and $E=\bigcap\limits_{m,n\in\mathbb N} E_{mn}$, the condition $(ii)$ is true.

$(ii)\Rightarrow (i)$. According to [7, Theorem 6.2.(1)] it enough to prove that the calculation function $f=c_X$ is a first Baire class function on $X\times C_p(X)$.

Let $a,b\in \mathbb R$, $a<b$ and $g:\mathbb R\to \mathbb R$ be a continuous function such that $g^{-1}(0)=(-\infty, a]\cup [b,+\infty)$. Clearly that the mappings $\varphi: C_p(X)\to C_p(X)$, $\varphi(y)(x)=g(y(x))$, and $h:X\times C_p(X)\to X\times C_p(X)$, $h(x,y)=(x,\varphi(y))$, are continuous. Since $f^{-1}((-\infty, a]\cup [b,+\infty)) = h^{-1}(E)$, the set $f^{-1}((a,b))$ is a functionally $F_{\sigma}$-set. Therefore according to [11, Theorem 3.8] the function $f$ is a first Baire class function.
\end{proof}

\begin{proposition}\label{p:2.2} Let $X$ be a topological space. Then the following conditions are equivalent:

$(i)$\,\,\, $X$ has $L$-property;

$(ii)$\,\,\,the set $E=\{(x,y): y(x)=0\}$ is a $G_{\delta}$-set in $X\times C_p(X)$.
\end{proposition}

\begin{proof} The implication $(i)\Rightarrow (ii)$ is obvious. The implication $(ii) \Rightarrow (i)$ can be proved analogously as in the previous proposition using [12, Proposition 2.3].
\end{proof}

\begin{remark}\label{r:2.3} It follows from Proposition 2.1 and 2.2 that the second part of Question 1.3 is a variant of the problem on the coincidence of Baire's and Lebesgue's classifications of separately continuous functions of two variables. According to [13, Theorem 4.12] the product of a family $(X_s:s\in S)$ of nontrivial separable linearly ordered spaces $X_s$ has $L$-property if and only if $|S|\leq 2^{\aleph_0}$. Therefore the compact space  $[0,1]^{[0,1]}$ is an example of a compact space with $L$-property which does not have $B$-property, because all compact spaces with $B$-property are metrizable. This gives the negative answer to the second part of Question 1.3.
\end{remark}

\begin{remark}\label{r:2.4} The space $X=[0,1]^{[0,1]}$ is an example of a space with $L$-property which is not a quarter-stratifiable space. This gives the negative answer to Question~1.4. Really, assuming that $X$ is a quarter-stratifiable space according to [7, Theorem 2.3], we obtain that $X$ is a metrically quarter-stratifiable space, hence has $B$-property.
\end{remark}

The following proposition refines a structure of spaces with $B$-property.

\begin{proposition}\label{p:2.5} Let $X$ be a completely regular space with $B$-property. Then every onepoint set is a $G_{\delta}$-set in $X$.
\end{proposition}

\begin{proof} Let $(f_n)^{\infty}_{n=1}$ be a sequence of continuous functions $f_n:X\times Y\to \mathbb R$, where $Y=C_p(X)$, which converges to the separately continuous calculation function $f=c_X$ pointwise on $X$ and $x_0\in X$. Taking into account the continuity of $f_n$ and the countable chin conditions of $Y$, for every $n\in\mathbb N$ we construct sequences $(U_{nm})_{m=1}^{\infty}$ of open neighborhood of $x_0\in X$ and $(V_{nm})_{m=1}^{\infty}$ of open sets in $Y$ such that the set $\bigcup\limits_{m=1}^{\infty} V_{nm}$ is dense in $Y$ and $|f_n(x',y')-f_n(x'',y'')| < \frac{1}{n}$ for every $m\in \mathbb N$, $x',x''\in U_{nm}$ and $y',y''\in V_{nm}$. Then $|f_n(x,y)-f_n(x_0,y)| \leq \frac{1}{n}$ for every $x\in \bigcap\limits_{m=1}^{\infty} U_{nm}$ and $y\in Y$. Thus, $f(x,y)=f(x_0,y)$ for every $x\in
\bigcap\limits_{n=1}^{\infty} \bigcap\limits_{m=1}^{\infty}U_{nm}$ and $y\in Y$. Since $X$ is a completely regular space, $\{x_0\} =\bigcap\limits_{n=1}^{\infty} \bigcap\limits_{m=1}^{\infty}U_{nm}$.
\end{proof}

\section{Quarter-stratifiable and $PP$-space}

In this section we show that metrically quarter-stratifiable spaces coincide with Hausdorff $PP$-spaces.

The notion of quarter-stratifiable space has the following equivalent reformulation (see [7, Theorem 1.4]) which is approximate to the definition of $PP$-space. A topological space $X$ is a quarter-stratifiable space if and only if there exist a sequence $({\mathcal U}_n)_{n=1}^{\infty}$ of open covers ${\mathcal U}_n=(U_{n,i}:i\in I_n)$ of $X$ and sequence $(\alpha_n)_{n=1}^{\infty}$ of families $\alpha_n=(x_{n,i}:i\in I_n)$ of points $x_{n,i}\in X$ such that for every $x\in X$ and neighborhood $U$ of $x$ in $X$ there exists $n_0\in\mathbb N$ such that $x_{n,i}\in U$ if $n\geq n_0$ and $x\in U_{n,i}$.

\begin{remark}\label{r:3.1} For metrically quarter-stratifiable space $X$ using the paracompactness of $(X,{\mathcal T})$ we cam choose ${\mathcal T}$-locally finite ${\mathcal T}$-open covers ${\mathcal V}_n=(V_{n,j}:j\in J_n)$ of $X$, which inscribed in the cover ${\mathcal U}_n$. Now putting $y_{n,j}=x_{n,i}$, where $i\in
I_n$ is an index such that $V_{n,j}\subseteq U_{n,i}$, and taking into account that all sets $V_{n,j}$ are functionally open in $X$, we obtain that every metrically quarter stratifiable space is a $PP$-space. Moreover, according to [14, Lemma 5.1.8], analogical reasoning show that the condition of local finiteness of partitions of the unit in the definition of $PP$-space is not essential.
\end{remark}

The following proposition show that every Hausdorff $PP$-space is metrically quarter-stratifiable.

\begin{proposition}\label{p:3.2} Let $X$ be a topological space and $(h_i:i\in I)$ be a partition of the unit on $X$. Then the function $p:X^2\to\mathbb R$, $p(x,y)=\sum\limits_{i\in I}|h_i(x)-h_i(y)|$, is a continuous pseudo-metric on $X$.
\end{proposition}

\begin{proof} Since $h_i(x)\geq 0$ for every $i\in I$ and $x\in X$ and $\sum\limits_{i\in I}h_i(x)=1$ for every $x\in X$, the function $p$ is correctly defined, moreover, $p(x,y)\leq 2$ for every $x,y\in X$.

It easy to see that $p$ satisfies the axioms od pseudo-metric. It remains to verify that $p$ is continuous. It enough to show that for every $x_0\in X$ and $\varepsilon>0$ the set $B_\varepsilon(x_0)=\{x\in X: p(x_0,x)<\varepsilon\}$ is a neighborhood of $x_0$ in $X$. We choose a finite set $I_0\subseteq I$ such that $\sum\limits_{i\in I_0}h_i(x_0)>1-\frac{\varepsilon}{4}$. Using the continuity of $h_i$ we found a neighborhood $U$ of $x_0$ in $X$ such that $\sum\limits_{i\in
I_0}|h_i(x_0)-h_i(x)|<\frac{\varepsilon}{4}$ for every $x\in U$. Then $\sum\limits_{i\in I_0}h_i(x)>1-\frac{\varepsilon}{2}$ and
$$
p(x_0,x)\leq \sum\limits_{i\in I_0}|h_i(x_0)-h_i(x)| +
\sum\limits_{i\in I\setminus I_0}h_i(x_0) + \sum\limits_{i\in
I\setminus I_0}h_i(x) < \frac{\varepsilon}{4} +
\frac{\varepsilon}{4} + \frac{\varepsilon}{2} = \varepsilon
$$
for every $x\in U$, that is $U\subseteq B_\varepsilon(x_0)$.
\end{proof}

\begin{proposition}\label{p:3.3} Let $X$ be a topological space and $((h_{n,i}:i\in I_n))_{n=1}^{\infty}$ be a sequence of partition of the unit on $X$. Then there exists a continuous pseudo-metric $p$ on $X$ such that all function $h_{n,i}$ are continuous with respect to $p$.
\end{proposition}

\begin{proof} For every $x,y\in X$ we put
$$
p(x,y)= \sum\limits_{n\in\mathbb N}\frac{1}{2^n}\sum\limits_{i\in
I_n} |h_{n,i}(x) - h_{n,i}(y)|.
$$
It follows from Proposition \ref{p:3.2} that $p$ is a continuous pseudo-metric on $X$. Moreover, $|h_{n,i}(x) - h_{n,i}(y)|\leq 2^n p(x,y)$ for every $n\in\mathbb N$, $i\in I_n$ and $x,y\in X$. Therefore all function $h_{n,i}$ are $p$-continuous.
\end{proof}

\begin{corollary}\label{c:3.4} Every Hausdorff $PP$-space is a metrically quarter-stratifiable.\end{corollary}

\section{Lindel$\rm\ddot{o}$f space with $B$-property}

In this section we study properties of Lindel$\rm\ddot{o}$f subsets of spaces with $B$-property. The notion of dependence on some coordinates of functions is a main technical tool in this investigation.

Let  $X\subseteq {\mathbb R}^S$, $Y$ be a set and $f:X\times Y\to \mathbb R$. We say that {\it $f$ concentrated on a set $T\subseteq S$ with respect to the first variable}, if $f(x',y)=f(x'',y)$ for every $x', x''\in X$ with $x'|_T = x''|_T$ and $y\in Y$; and {\it $f$ depends on $\aleph$ coordinates with respect to the first variable}, where $\aleph$ is an infinite cardinal, if there exists a set $T\subseteq S$ such that $f$ concentrated on $T$ and $|T|\leq\aleph$. The notions of dependence with respect to the second variable or dependence of mapping $f:X\to\mathbb R$ can be introduced analogously.

Clearly that for a Lindel\"{o}f space $X\subseteq {\mathbb R}^S$ every continuous function $f:X\to \mathbb R$ depends on at most countable quantity of coordinates.

The next result take an important place in the study of Lindel\"{o}f subsets od spaces with $B$-property.

\begin{theorem}\label{th:4.1} Let $X\subseteq \mathbb R^S$ be a Lindel\"{o}f space, $Z$ be a completely regular space, $Y=C_p(Z)$, $B\subseteq Y$ be a dense in $Y$ set and $f:X\times Y\to\mathbb R$ be a function which is continuous with respect to the second variable and jointly continuous at every point of set $X\times B$. Then $f$ depends on countable quantity of coordinates with respect to the first variable.
\end{theorem}

\begin{proof} Suppose the contrary, that is for every at most countable set $T\subseteq S$ there exist $x',x''\in X$ and $y\in Y$ such that $x'|_T=x''|_T$ and $f(x',y)\ne(x'',y)$. Note that since $f$ is continuous with respect to the second variable and $B$ is dense in $Y$, without loss of generality we can propose that $y\in B$.

For every $y\in B$ and $\varepsilon >0$ using the joint continuity of $f$ at every point of set $X\times\{y\}$ and the Lindel\"{o}f property of $X$ we find a cover $(U(y,\varepsilon,n): n\in\mathbb N)$ of $X$ by open basic sets $U(y,\varepsilon,n)$ and a sequence $(V(y,\varepsilon,n): n\in\mathbb N)$ of neighborhoods of $y$ in $Y$ such that $|f(x',y') - f(x'',y'')|<\varepsilon$ for every $n\in\mathbb N$ ³ $(x',y'),(x'',y'')\in U(y,\varepsilon,n)\times V(y,\varepsilon,n)$. Using the structure of basic sets in the space $X\subseteq {\mathbb R}^S$ we choose a countable set $T(y,\varepsilon)\subseteq S$ such that for every $n\in\mathbb N$ and $x',x''\in X$ the conditions $x'|_{T(y,\varepsilon)} = x''|_{T(y,\varepsilon)}$ and $x'\in U(y,\varepsilon,n)$ imply $x''\in U(y,\varepsilon,n)$. Put $T(y)=\bigcup\limits_{m=1}^{\infty} T(y,\frac{1}{m})$. Clearly that $T(y)$ is countable, in particular, the function $f_y:X\to\mathbb R$, $f_y(x)=f(x,y)$, concentrated on $T(y)$.

Let $\omega_1$ is the first uncountable ordinal. Using the transfinite induction we construct an increasing sequence $(T_{\alpha}:\alpha<\omega_1)$ of countable sets $T_{\alpha}\subseteq S$, a sequence $(\varepsilon_{\alpha}:\alpha<\omega_1)$ of reals $\varepsilon_{\alpha}>0$, sequences $(x'_{\alpha}:\alpha<\omega_1)$, $(x''_{\alpha}:\alpha<\omega_1)$ and $(y_{\alpha}:\alpha<\omega_1)$ of $x'_{\alpha}, x''_{\alpha}\in X$ and $y_{\alpha}\in Y$ and a sequence $(V_{\alpha}:\alpha<\omega_1)$ of neighborhoods $V_{\alpha}$ of $y_{\alpha}$ in $Y$ such that for every $\alpha<\omega_1$ the following conditions are satisfied:

$(a)$\,\,\,$x'_{\alpha}|_{T_{\alpha}} = x''_{\alpha}|_{T_{\alpha}}$;

$(b)$\,\,\,$T(y_{\alpha})\subseteq T_{\alpha+1}$;

$(c)$\,\,\,$|f(x'_{\alpha},y) - f(x''_{\alpha},y)|> \varepsilon_{\alpha}$ for every $y\in V_{\alpha}$.

Let $T_1\subseteq S$ is a countable set. Using the assumption we choose $x'_1, x''_1 \in X$ and $y_1\in B$ such that $f(x'_1, y_1)\ne f(x''_1, y_1)$. Put $\varepsilon_1 =\frac{1}{2} |f(x'_1, y_1) -  f(x''_1, y_1)|$. Using the continuity of $f$ with respect to the second variable we choose a neighborhood $V_1$ of $y_1$ in $Y$ such that the condition $(c)$ holds for $\alpha =1$.

Suppose that the sequences $(T_{\alpha}:\alpha<\beta)$, $(\varepsilon_{\alpha}:\alpha<\beta)$, $(x'_{\alpha}:\alpha<\beta)$, $(x''_{\alpha}:\alpha<\beta)$,
$(y_{\alpha}:\alpha<\beta)$ and $(V_{\alpha}:\alpha<\beta)$, where $\beta<\omega_1$, satisfy the conditions $(a)$, $(b)$ and $(c)$. If $\beta$ is a limited ordinal, then we put $T_{\beta}=\bigcup\limits_{\alpha<\beta} T_{\alpha}$. If $\beta=\gamma +1$ for some at most countable ordinal $\gamma$, then we put $T_{\beta}= T_{\gamma}\cup T(y_{\gamma})$. Further using the assumption analogously as for $\alpha = 1$ we find $x'_{\beta}, x''_{\beta} \in X$, $y_{\beta}\in B$, $\varepsilon_{\beta}> 0$ and a neighborhood $V_{\beta}$ of $y_{\beta}$ in $Y$ such that the conditions $(a)$ and $(c)$ hold.

For every $\alpha<\omega_1$ we choose a finite set $C_{\alpha}\subseteq Z$ and $\delta_{\alpha}>0$ such that $\tilde{V}_{\alpha} = \{y\in Y: |y(z)-y_{\alpha}(z)| <
\delta_{\alpha}$ for every $z\in C_{\alpha}\} \subseteq V_{\alpha}$. Since $\aleph_1 = |\omega_1|$ is a regular cardinal, using Shanin's Lemma [14, p.185] we obtain that there exist $\varepsilon_0 >0$, $\delta_0 >0$, a finite set $C\subseteq Z$ and a set $\Gamma \subseteq [1, \omega_1)$ such that $|\Gamma|=\aleph_1$, $\varepsilon_{\gamma} \geq \varepsilon_0$ and $\delta_{\gamma}\geq \delta_0$ for every $\gamma \in \Gamma$ and $C_{\gamma'} \cap C_{\gamma''} = C$ for every distinct $\gamma',\gamma'' \in \Gamma$. We consider the continuous mapping $\varphi: \{y_{\gamma}:\gamma\in\Gamma\}\to \mathbb R^C$, $\varphi(y_{\gamma}) = (y_{\gamma}(z))_{z\in C}$. Since $\mathbb R^C$ is a separable metrizable space, there exists $\gamma_0 \in \Gamma$ such that $|\Gamma_0|=\aleph_1$, where $\Gamma_0=\{\gamma\in \Gamma: |y_{\gamma}(z) -
y_{\gamma_0}(z)|<\delta_0$ for every $z\in C\}$.

Show that for every neighborhood $V$ of $y_0=y_{\gamma_0}$ the set $\Gamma(V)=\{\gamma\in \Gamma_0: V\cap \tilde{V_{\gamma}}= \O\}$ is finite. Let $C'\subseteq Z$ is a finite set, $\delta'> 0$ and $V=\{y\in Y: |y(z)-y_0(z)|<\delta'$ for every $z\in C'\}$. Since $Z$ is a completely regular space, the condition $V\cap \tilde{V_{\gamma}} = \O$ for some $\gamma\in \Gamma_0$ implies the existence of $z_{\gamma}\in C'\cap C_{\gamma}$ such that $|y_{\gamma}(z) - y_0(z)| \geq \delta_{\gamma} + \delta' >
\delta_0$. It follows from the definition of the set $\Gamma_0$ that $z_{\gamma}\not\in C$. Taking into account that $C_{\gamma'}\cap C_{\gamma''} = C$ for every distinct $\gamma', \gamma'' \in \Gamma_0$ we obtain that $z_{\gamma'}\ne z_{\gamma''}$ for every distinct $\gamma', \gamma''\in \Gamma(V)$. Thus, $|\Gamma(V)| \leq |C'\setminus C| < \aleph_0$.

Choose $m_0\in \mathbb N$ such that $\frac{1}{m_0}<\varepsilon_0$. Since the set $\Gamma' = \bigcup\limits_{n\in\mathbb N} \Gamma(V(y_0,\frac{1}{m_0},n))$ is at most countable, $|\Gamma_0\setminus \Gamma'| = \aleph_0$. Therefore there exist $\beta\in\Gamma_0$ such that $\beta > \gamma_0$ and $V(y_0,\frac{1}{m_0},n)\cap \tilde{V}_{\beta} \ne \O$ for every $n\in \mathbb N$. Recall that $(U(y_0,\frac{1}{m_0}):n\in\mathbb N)$ is a cover of $X$. We take $n_0\in \mathbb N$ such that $x'_{\beta}\in U(y_0, \frac{1}{m_0} n_0))$. Since $(b)$ implies that $T(y_0,\frac{1}{m_0})\subseteq T(y_0) = T(y_{\gamma_0})\subseteq T_{\beta}$ and  $x'_{\beta}|_{T_{\beta}} = x''_{\beta}|_{T_{\beta}}$ according to $(a)$, $x''_{\beta}\in U(y_0,\frac{1}{m_0}, n_0)$. Now take $y\in
\tilde{V}_{\beta}\cap V(y_0,\frac{1}{m_0},n_0)$. According to $(c)$ we obtain
$$
|f(x'_{\beta}, y) - f(x''_{\beta}, y)| > \varepsilon_{\beta} \geq
\varepsilon_0.
$$
On other hand, according to the choice of the sets $U(y_0,\frac{1}{m_0},n_0)$ and $V(y_0,\frac{1}{m_0},n_0)$, we have
$$
|f(x'_{\beta}, y) - f(x''_{\beta}, y)| < \frac{1}{m_0} <
\varepsilon_0.
$$

Hence, we obtain a contradiction and the Theorem is proved.\end{proof}

The following result gives the positive answer to the Question 1.5.

\begin{theorem}\label{th:4.2} Let $Z$ be a completely regular space with $B$-property and $X\subseteq Z$ be a Lindel$\rm\ddot{o}$f subset of $Z$. Then there exist a separable metrizable space $H$ and a continuous bijection $\varphi:X\to H$.
\end{theorem}

\begin{proof} Without loos of generality we can assume that $Z\subseteq \mathbb R^S$. Put $Y=C_p(Z)$ and we consider the separately continuous function $g:Z\times Y\to\mathbb R$, $g(z,y)=y(z)$. Since $Z$ has $B$-property, there exists a sequence of continuous function $g_n: Z\times Y\to \mathbb R$ such that $g(z,y) = \lim\limits_{n\to\infty} g_n(z,y)$ for every $(z,y)\in Z\times Y$. Put $f_n = g_n|_{X\times Y}$. According to Theorem 4.1, for every $n\in\mathbb N$ there exists an at most countable set $T_n\subseteq S$ such that $f_n$ concentrated on $T_n$ with respect to the first variable. Then the function $f$ concentrated on the set $T =
\bigcup\limits_{n=1}^{\infty} T_n$ as the poinwise limit of the sequence $(f_n)^{\infty}_{n=1}$.

It remains to consider the continuous mapping $\varphi: X\to \mathbb R^T$, $\varphi(x)=X|_T$. Note that $\varphi$ is a bijection with valued in the separable metrizable space $H=\varphi(X)$. Really if $\varphi(x_1)=\varphi(x_2)$, then $y(x_1)=y(x_2)$ for every $y\in Y$. Therefore $x_1=x_2$, because $Z$ is completely regular.
\end{proof}

\begin{remark}\label{r:4.3} Every countable space $X$ has $B$-property, because the space $C_p(X)$ is metrizable. Therefore every nonmetrizable countable space is an example which shows that Theorem 4.2 can not be strengthened to metrizability of Lindel$\rm\ddot{o}$f subspace of spaces with $B$-property (analogously as for compact spaces).
\end{remark}

\section{Namioka spaces and Maran spaces}

We will use the following two auxiliary results from [15], which give a possibility to use the technic of the dependence on some coordinates for the obtaining the Namioka property.

\begin{proposition}\label{p:5.1} Let $Y\subseteq {\mathbb R}^T$ be a compact space, $(Z,|\cdot - \cdot|)$ be a metric space, $f:Y\to Z$ be a continuous mapping $\varepsilon \geq 0$ and set $S\subseteq T$ such that $|f(y')-f(y'')|_Z\leq\varepsilon$ for every $y',y''\in Y$ with $y'|_S=y''|_S$. Then for every $\varepsilon'>\varepsilon$ there exists a finite set $S_0\subseteq S$ and $\delta>0$ such that $|f(y')-f(y'')|_Z\leq \varepsilon'$ for every $y',y''\in Y$ with $|y'(s)-y''(s)|<\delta$ for every $s\in S_0$.
\end{proposition}

\begin{proposition}\label{p:5.2} Let $X$ be a Baire spase, $Y\subseteq {\mathbb R}^T$ be a compact space, $f:X\times Y\to \mathbb R$ be a separately continuous function. Then the following statement are equivalent:

$(i)$ $f$ has the Namioka property;

$(ii)$ for every open in $X$nonempty set $U$ and a real $\varepsilon >0$ there exist an open in $X$ nonempty set $U_0\subseteq U$ and an at most countable set $S_0\subseteq T$ such that $|f(x,y')-f(x,y'')|\leq \varepsilon$ for every $x\in U_0$ and $y',y''\in Y$ with $y'|_{S_0}=y''|_{S_0}$.
\end{proposition}

Let $X$ be a topological space.Define the Shoquet game on $X$ in which two players $\alpha$ and $\beta$ participate. A nonempty open in $X$ set $U_0$
is the first move of $\beta$ and a nonempty open in $X$ set $V_1\subseteq U_0$ is the first move of $\alpha$. Further $\beta$ chooses a nonempty open in $X$
set $U_1\subseteq V_1$ and $\alpha$ chooses a nonempty open in $X$ set $V_2\subseteq U_1$ and so on. The player $\alpha$ wins if $\bigcap\limits_{n=1}^{\infty}V_n\ne\O$. Otherwise $\beta$ wins.

A topological space $X$ is called {\it $\alpha$-favorable} if $\alpha$ has a winning strategy in this game. A topological space $X$ is called {\it $\beta$-unfavorable} if $\beta$ has no winning strategy in this game. Clearly, any $\alpha$-favorable topological space $X$ is a $\beta$-unfavorable space. It was shown in [6] that a topological game $X$ is Baire if and only if $X$ is $\beta$-unfavorable.

It is well-known (see [17]) that $\beta$-unfavorability of $X$ is equivalent to the fact that $X$ is a Baire space.

The following result occupies a central place in this section.

\begin{theorem}\label{th:5.3} Let $X$ be a Baire space, $Y\subseteq \mathbb R^T$ be a compact space and $f:X\times Y\to\mathbb R$ be a separately continuous Baire measurability function. Then $f$ has the Namioka property.
\end{theorem}

\begin{proof} Since $f$ is a Baire measurability, there exists an countable family ${\mathcal F}$ of continuous functions on $X\times Y$ such that the function  $f$ can be obtained as at most countable times pointwise limit of sequence of functions from ${\mathcal F}$. Let ${\mathcal F} = (f_n: n\in\mathbb N)$.

Suppose that $f$ has not the Namioka property. Then $|T|>\aleph_0$ and according to Proposition 5.2, there exist an open in $X$ nonempty set $U_0$ and  $\varepsilon_0>0$ such that for every open in $X$ nonempty set $U\subseteq U_0$ and at most countable set $S\subseteq T$ there exist $x\in U$ and $y',y''\in Y$
such that $y'|_S=y''|_S$ ³ $|f(x,y') - f(x,y'')| >\varepsilon_0$.

We construct a strategy $\tau$ for the player $\beta$ in the Shoquet game on the space $X$. The set $U_0$ is the first move of $\beta$. Let $V_1\subseteq U_0$ is a nonempty in $X$ set which is the first move of the player $\alpha$. Since the continuous function $f_1$ has the Namioka property, according to Proposition 5.2 there exist an open in $X$ nonempty set $\tilde{U}_1\subseteq V_1$ and countable set $S_1\subseteq T$ such that 
$$
|f_1(x,y')-f_1(x,y'')| < 1
$$
for every $x\in \tilde{U}_1$ and $y',y''\in Y$ with $y'|_{S_1}=y''|_{S_1}$. Using the assumption we find points $x_1\in\tilde{U}_1$ and $y'_1,y''_1\in Y$ such that $y'_1|_{S_1}=y''_1|_{S_1}$ and
$$
|f(x_1,y'_1)-f(x_1,y''_1)| > \varepsilon_0.
$$
Since $f$ is continuous with respect to the second variable, there exists an open in $X$ neighborhood $U_1\subseteq \tilde{U}_1$ of $x_1$ such that 
$$
|f(x,y'_1)-f(x,y''_1)| > \varepsilon_0
$$
for every $x\in U_1$. Now we put $U_1=\tau(U_0,V_1)$, that is $U_1$ is the second move of $\beta$.

Further, let $V_2\subseteq U_1$ be a nonempty open in $X$ set which is the second move of $\alpha$. Choose an open in $X$ nonempty set $\tilde{U}_2\subseteq V_2$ and countable set $S_2\subseteq T$ which contains the set $S_1$ such that
$$
|f_2(x,y')-f_2(x,y'')| < \frac{1}{2}
$$
for every $x\in \tilde{U}_2$ and $y',y''\in Y$ with $y'|_{S_2}=y''|_{S_2}$. Using the assumption and the continuity of $f$ with respect to the first variable we choose an open in $X$ nonempty set $U_2\subseteq \tilde{U}_2$ and points $y'_2, y''_2\in Y$ such that $y'_2|_{S_2}=y''_2|_{S_2}$ and 
$$
|f(x,y'_2)-f(x,y''_2)| > \varepsilon_0
$$
for every $x\in U_2$, and put $U_2 = \tau(U_0, V_1, U_1, V_2)$.

Continuing this process to infinity we obtain sequences $(U_n)_{n=1}^{\infty}$ and $(V_n)_{n=1}^{\infty}$ of open in $X$ nonempty sets $U_n$ and $V_n$, $(y'_n)_{n=1}^{\infty}$ and $(y''_n)_{n=1}^{\infty}$ of points $y'_n, y''_n \in Y$ and an increasing sequence $(S_n)_{n=1}^{\infty}$ of countable sets 
$S_n\subseteq T$ such that for every $n\in\mathbb N$ the following conditions:

$(a)$\,\,\,\,$U_n\subseteq V_n\subseteq U_{n-1}$;

$(b)$\,\,\,\,$y'_n|_{S_n} = y''_n|_{S_n}$;

$(c)$\,\,\,\,$|f_n(x,y')-f_n(x,y'')| < \frac{1}{n}$ for every $x\in U_n$ and $y',y''\in Y$ with $y'|_{S_n}=y''|_{S_n}$;

$(d)$\,\,\,\,$|f(x,y'_n)-f(x,y''_n)| > \varepsilon_0$ for every $x\in U_n$;

\noindent are true.

Since the space $X$ is Baire, the strategy $\tau$ is not a winning strategy for $\beta$. Thus, there exist a game in which the player loses playing according to $\tau$. In other words, there exist corresponding sequences with the properties $(a)$, $(b)$,$(c)$ and $(d)$ such that $\bigcap\limits_{n=1}^{\infty} U_n \ne \O$.

Take a point $x_0\in \bigcap\limits_{n=1}^{\infty} U_n$ and put $S_0 = \bigcup\limits_{n=1}^{\infty} S_n $. Let a function $g: X\times Y\to\mathbb R$ is the pointwise limit of a subsequence $(f_{n_k})_{k=1}^{\infty}$ of sequence $(f_n)_{n=1}^{\infty}$. Taking into account that $S_{n_k}\subseteq S_0$ and 
$x_0\in U_{n_k}$ for every $k\in \mathbb N$, and going to the limit in the condition $(c)$ by $k\to\infty$, we obtain that 
$g(x_0,y')=g(x_0,y'')$ for every $y',y''\in Y$ with $y'|_{S_0}=y''|_{S_0}$. Thus, for every pointwise limit $g$ of a sequence $(f_{n_k})_{k=1}^{\infty}$ its vertical section $g^{x_0}:Y\to\mathbb R$, $g^{x_0}(y)=g(x_0,y)$, concentrated on $S_0$. Since pointwise limit of a sequence of functions $g_n:Y\to\mathbb R$, which concentrated on set $S_0$, concentrated on $S_0$ too, the vertical section $f^{x_0}:Y\to\mathbb R$, $f^{x_0}(y)=f(x_0,y)$, of the function $f$ concentrated on $S_0$. According to Proposition 5.1 there exists a set $T_0\subseteq S_0$ such that $|f(x_0,y')-f(x_0,y'')| \leq \varepsilon_0$ for every $y',y''\in Y$ with $y'|_{T_0}=y''|_{T_0}$. Since the sequence $(S_n)_{n=1}^{\infty}$ increase, there exists an integer $n_0\in \mathbb N$ such that $T_0\subseteq S_{n_0}$. Then it follows from the condition $(b)$ and the choice $T_0$ that $|f(x_0,y'_{n_0})-f(x_0,y''_{n_0})| \leq\varepsilon_0$. On other hand, according to $(d)$, we have $|f(x_0,y'_{n_0})-f(x_0,y''_{n_0})| >\varepsilon_0$, a contradiction.
\end{proof}

The following corollary gives the positive answer to Question 1.2.

\begin{corollary} Every Baire weakly Moran space is a Namioka space.
\end{corollary}

\bibliographystyle{amsplain}

\end{document}